\documentclass[11pt]{amsart}
\usepackage{amsmath,amssymb,amsfonts,amscd}
\usepackage{enumerate}
\usepackage{amsrefs}

\numberwithin{equation}{section}
\theoremstyle{plain}
\newtheorem{thm}{Theorem}
\newtheorem{prop}{Proposition}[section]
\newtheorem{lemma}[prop]{Lemma}
\newtheorem{cor}[prop]{Corollary}
\newtheorem{conj}[prop]{Conjecture}

\theoremstyle{definition} \theoremstyle{definition}
\newtheorem{defn}[prop]{Definition}
\newtheorem{rem}[prop]{Remark}
\newtheorem*{ack}{Acknowledgments}

\newcommand{\doubleslash}{/\!\!/}

\newcommand{\Mod}{\text{-Mod}}
\newcommand{\Hecke}{\mathcal{H}}
\newcommand{\RR}{\mathcal{R}}
\newcommand{\I}{{\mathcal{I}}}
\newcommand{\OO}{\mathcal{O}}
\newcommand{\g}{\mathfrak{g}}

\newcommand{\ellip}{\mathrm{ell}}
\newcommand{\fraks}{\mathfrak{s}}
\newcommand{\Z}{\mathbb{Z}}
\newcommand{\R}{\mathbb{R}}
\newcommand{\F}{\mathbb{F}}

\newcommand{\C}{\mathbb{C}}

\newcommand{\PP}{\mathbb{P}}

\newcommand{\M}{\mathbb{M}}

\DeclareMathOperator{\Ext}{Ext}
\DeclareMathOperator{\Hom}{Hom}
\DeclareMathOperator{\End}{End}
\DeclareMathOperator{\Ind}{Ind}
\DeclareMathOperator{\ind}{ind}
\DeclareMathOperator{\diag}{diag}

\def\SL{{\rm SL}}
\def\Ps{{\rm Ps}}
\def\EP{{\rm EP}}
\def\GSp{{\rm GSp}}

\def\GL{{\rm GL}}

\newcommand{\set}[2]{ 
	{\left\{\left.
	#1\vphantom{#2\bigl(\bigr)}\,\right|
	\,#2\right\}}}

\begin{document}

\title{Extensions of representations of $p$-adic groups}

\author{Jeffrey D. Adler}
\address{
American University \\
Washington, DC  20016-8050\\
USA} 
\email{jadler@american.edu}

\author{Dipendra Prasad}
\address{
Tata Institute of Fundamental Research \\
Homi Bhabha Road \\
Mumbai - 400 005 \\
INDIA}
\email{dprasad@math.tifr.res.in}


\date{\today}

\begin{abstract}
We calculate extensions between certain irreducible admissible 
representations of $p$-adic groups.
\end{abstract}

\maketitle

\centerline{\sf To Hiroshi Saito, in memoriam}
\setcounter{tocdepth}{1}

\tableofcontents

\section{Introduction}
The classification of irreducible admissible representations of groups over local fields has been a very 
active and successful branch of mathematics. One next step in the subject 
would be to understand all possible
extensions between irreducible representations.
Many results of a general
kind are known about extensions between admissible representations of $p$-adic groups, most 
notably the notion of the Bernstein center and many other results of Bernstein and Casselman.
These results 
reduce the question to one between components of one parabolically induced representation,
cf.\ Lemma \ref{lem:cuspidal-support} below.
Specific calculations seem not to have attracted attention
except for
$\Ext^i_G(\C,\C)$, 
which is the cohomology $H^i(G,\C)$ of $G$ in terms of measurable cochains;
besides these, extensions of generalized Steinberg representations are studied in \cites{dat,orlik}.
In this paper,
we calculate $\Ext^i_G(\pi_1,\pi_2)$, abbreviated to $\Ext^i(\pi_1,\pi_2)$,
between certain irreducible admissible 
representations $\pi_1,\pi_2$ 
of $G = {\sf G}(k)$ where ${\sf G}$ is a connected 
reductive algebraic group over a non-archimedean local field $k$ of characteristic 0; we 
abuse notation in the usual way and call $G$ itself a connected reductive algebraic group.

Since extensions of representations of 
abelian groups are well understood through the cohomology $H^i(\Z^n,\C)$ of $\Z^n$,
it is no loss of generality when considering extensions $\Ext^i(\pi_1,\pi_2)$ to restrict oneself
to the subcategory $\RR^\chi(G)$ of the category $\RR(G)$ of all smooth representations
of $G$, consisting of those representations on which the center of $G$
acts via a given character $\chi$, which we can also assume to be unitary.

We have two main results.
The first is as follows.

\begin{thm}
\label{thm:main}
Let $G$ be a reductive $p$-adic group over $k$, and $P$ a maximal $k$-parabolic
subgroup of $G$ with Levi decomposition $P=MN$.
Let $\sigma$ be an irreducible,
supercuspidal representation
of $M$,
and let
$\pi=i_{P}^{G}\sigma$,
where $i_{P}^{G}$ denotes normalized induction.
If $\pi$ is irreducible, then
$$
\Ext^1_{\RR^\chi(G)}(\pi,\pi) = \C.
$$
If $\pi$ is reducible,
then it has two inequivalent, irreducible subquotients.
Let $\pi_1$ and $\pi_2$ denote these two subquotients.
Then 
$$
\Ext^1_{\RR^\chi(G)}(\pi_i,\pi_j) =
\begin{cases}
0 & \text{if $i =  j$}, \\
\C & \text{if $i\neq j$}.
\end{cases}
$$
\end{thm}

\begin{rem}
The theorem is an extension of an observation that one of the authors
made concerning reducible unitary principal series representations of 
$\GSp_4(k)$ arising from the Klingen parabolic
(see \cite{bessel}*{Remark 11.2}),
prompting a similar question for $\SL_2(k)$, which on checking around we found was not known.

A similar statement is true for $(\g,K)$-modules for
representations $\pi_1,\pi_{-1}$ of $\SL_2(\R)$ of weights 
$1,-1$ respectively, as follows by looking 
at the complete list of indecomposable representations of $\SL_2(\R)$ supplied
by Howe-Tan, cf.\ \cite{howe-tan:sl2}*{Theorem II.1.1.13}.
\end{rem}

Our second result concerns the components of certain principal series representations
of $\SL_n(k)$.
Suppose $\omega \colon k^\times \longrightarrow \C^\times$ is a character of order $n$.
We assume that $\omega$ is either unramified or totally ramified,
in the sense that the restriction of $\omega$
to the group $\OO^\times$ of units in $k^\times$
is either trivial or has order $n$.
Let $\pi$ be the principal series representation
$\Ps(1,\omega,\cdots,\omega^{n-1})$ of $\GL_n(k)$, as well as its restriction
to $\SL_n(k)$ which is known to decompose into a direct sum of $n$ inequivalent,
irreducible, admissible representations of $\SL_n(k)$, permuted transitively by the action
of $\GL_n(k)$ on $\SL_n(k)$ by conjugation.
Embed $k^\times$ inside $\GL_n(k)$
as the group of upper left diagonal matrices with all other diagonal entries 1.
This $k^\times$ too acts transitively on the set of 
irreducible summands of the representation $\pi$ of $\SL_n(k)$;
call $\pi_1$ one of them.
Then the set of irreducible representations of $\SL_n(k)$
appearing in $\pi$ can be indexed as $\pi_e$ for $e$ belonging to $k^\times$,
in fact more precisely for $e$ belonging to $k^\times/\ker(\omega)$ since it is
known that elements of $k^\times$ belonging to $\ker(\omega)$ act 
trivially on $\pi_1$. 
For the statement of the next theorem, 
the choice of the base point representation $\pi_1$ plays no role, but 
this indexing of representations
occurring in $\pi$ through $k^\times/\ker(\omega)$ is important.

Let $S_\pi$ be the group of characters of $k^\times$ generated by $\omega$, i.e.,
$S_\pi = \{ 1,\omega,\cdots, \omega^{n-1}\}$. Then the character group $\widehat{S}_\pi$ of $S_\pi$
can be identified to $k^\times/{\ker (\omega)}$ via the natural pairing
\begin{align*}
S_\pi \times k^\times/{\ker (\omega)}  &\longrightarrow  \C^\times,   \\ 
(\chi,x) &\longmapsto \chi(x).
\end{align*} 
Let $X = \C[S_\pi]$ be the group ring
of $S_\pi$, and $Y = \C[S_\pi]^0$ be the augmentation ideal of $\C[S_\pi]$. Then $Y$, and hence
$\Lambda^i Y$, are representation spaces of $S_\pi$, and it makes sense to talk 
of $\Lambda^i Y [e]$, the $e$-th
isotypic component of $\Lambda^iY$, for $e$ a character of $S_\pi$, which as mentioned
earlier can be identified to $k^\times/\ker(\omega)$. 

\begin{thm}
\label{thm:sln}
With the notation as above, for $a,b \in k^\times/\ker(\omega)$, we have
$$
\Ext^r(\pi_a,\pi_b) \cong \Lambda^r Y [ba^{-1}].
$$
In particular, 
$$
\Ext^1(\pi_a,\pi_b) = \C \quad\text{if}\quad  a \not = b, \qquad\text{and}\qquad \Ext^1(\pi_a,\pi_a)=0.
$$
\end{thm}

Of course, when $n=2$, this is 
a special case of Theorem \ref{thm:main},
and thus we have two different computations of extensions
of representations of $\SL_2(k)$.
Neither is trivial.
One uses Kazhdan's orthogonality criterion,
and the other uses
nontrivial statements about Hecke algebras.

We might add that most of the paper is devoted to the proof of the Theorem \ref{thm:sln}, which is divided
into two separate cases depending on whether the character $\omega$ is totally ramified, or is unramified. 
In both of these cases, the question about calculating the Ext groups is turned into one about 
modules over appropriate Hecke algebras, then
to modules over certain group algebras, and finally to questions about cohomology of groups.
Although in these two cases the Hecke algebras involved are quite different, at the end the questions
boil down to the same calculation about
$$
\Ext^1_{A\rtimes \Z/n}(\chi_1,\chi_2),
$$
where $A$ is the group $A = \set{(k_1, \cdots,k_n) \in \Z^n }{\sum k_i = 0}$, with the cyclic 
permutation action of $\Z/n$ on it, and $\chi_1,\chi_2$ are characters of $A\rtimes \Z/n$. 

Our Theorem \ref{thm:sln} for a very special class of principal series representations
of $\SL_n(k)$ begs for a formulation more generally; we offer a conjecture for $\SL_n(k)$:

\begin{conj}
Let $\pi$ be an irreducible unitary principal series representation of $\GL_n(k)$ induced from
a supercuspidal representation $\sigma$ of a Levi subgroup $M$.
Define,
$$
S_\pi= \{\mu \mid \pi \otimes \mu \cong \pi\},
$$
where $\mu$ ranges over the set of complex characters of $k^\times$
(considered as characters of $\GL_n(k)$ via the determinant map).
Similarly, define,
$$
S_\sigma=  \{\mu \mid \sigma \otimes \mu \cong \sigma\}.
$$
Clearly $S_\sigma \subset S_\pi$, and it is easy to see that 
$S_\pi/{S_\sigma}$ is a subgroup of the Weyl group
$W(\GL_n(k),M)= N_{\GL_n(k)}(M)/M$.
Let $Y$ be the character group 
of $SM= M \cap \SL_n(k)$ 
which is a module for
$W(\GL_n(k),M)$,
and in particular for $S_\pi/S_\sigma$.
Then characters of $S_\pi$ 
---parametrized just as before by a quotient, say $Q$, of $k^\times$ ---determine 
irreducible representations of $\SL_n(k)$ contained in $\pi$, whereas
characters of $S_{\pi}/{S_\sigma}$ determine irreducible representations on $\SL_n(k)$ contained 
in a principal series representation, say $\pi_0$, of $\SL_n(k)$ induced from an irreducible component, 
say $\sigma_0$,
of 
$\sigma$ restricted to $SM= M \cap \SL_n(k)$. 
For $a,b \in Q$, we conjecture that:
$$
\Ext^r(\pi_a,\pi_b) \cong  \Lambda^r Y [ba^{-1}].
$$
\end{conj}

\begin{rem}
We recall that in the $L$-packet of $\SL_n(k)$ determined by $\pi$, 
there is a further partitioning
depending on whether the representations belong to the same Bernstein component or not:
this is the difference between $S_\pi$ which determines the $L$-packet and $S_\pi/{S_\sigma}$ which determines 
the part of the $L$-packet in a given Bernstein component. The above conjecture includes the statement 
that unless $\pi_a$ and $\pi_b$ belong to the same Bernstein component, all the Ext groups are zero.
\end{rem}

\begin{rem}
Although we appeal to existing knowledge about structure of Hecke algebras to prove
Theorem \ref{thm:sln},
some details of the equivalence of the category of representations of $p$-adic groups
versus those of the Hecke algebra are necessary since to convert the problem about representations of
$p$-adic groups to one on Hecke modules,
we must know what are the corresponding objects on the Hecke algebra side.
It is possible 
sometimes to come up with the suggested objects on the Hecke algebra with 
pure thought---for example for
Theorem \ref{thm:sln} in the totally ramified case, these will be exactly those representations of the 
Hecke algebra which are of dimension 1,
and there are exactly $n$ of them corresponding to components of the principal series representation
$\Ps(1,\omega,\cdots,\omega^{n-1})$. We have however preferred to identify the modules of the Hecke algebra 
concretely, and in the process have tried to give an exposition on what goes into it for the 
benefit of some of the readers, as well as for the authors. 
\end{rem}

\begin{ack}
We thank the National Science Foundation for financial support (DMS-0854844),
and the Department of Mathematics and Statistics at American University
for its hospitality.
We thank Alan Roche for many helpful conversations and correspondence on Hecke algebras,
and specially on clarifications on his own work 
with David Goldberg on `types' for
principal series representations of $\SL_n$, in particular for the proof of Proposition \ref{prop:sln-hecke}.
We thank
Gordan Savin for the construction in Section~\ref{sec:savin}
of a nontrivial extension of representations of $\SL_2(k)$.
Finally, we thank an anonymous referee for helpful comments.

After the first draft of this paper was written, we saw the recent preprint
of Opdam and Solleveld \cite{opdam-solleveld:extensions}.
Our results should emerge as special cases of theirs
once appropriate identifications are made,
a process that would require some work.
However, our proofs are quite different from theirs. The authors thank Opdam and
Solleveld for their comments in this regard.
\end{ack}

\section{Preliminary results for Theorem \ref{thm:main}}

Given a connected reductive $k$-group $G$, and
two admissible, finite-length
representations $\pi$ and $\pi'$
of $G$ having a given central character, 
one can consider
the \emph{Euler-Poincare pairing}
between $\pi$ and $\pi'$,
which is denoted $\EP(\pi,\pi')$, 
and defined by
$$
\EP(\pi,\pi') = \sum _i (-1)^i\dim_\C \Ext^i(\pi,\pi').
$$
Here, each $\Ext^i(\pi,\pi')$ is a finite-dimensional vector space over $\C$,
and is zero when $i$ is greater than the $k$-split rank of $G/Z(G)$.
The notion of the Euler-Poincare pairing and its usefulness in the context of $p$-adic groups,
especially property (\ref{item:EP-chars}) 
of Proposition \ref{prop:EP} below,
was noted by Kazhdan in
\cite{kazhdan:cuspidal}.
One can find a proof by Schneider and Stuhler
\cite{schneider-stuhler:sheaves} in characteristic zero for
property (\ref{item:EP-chars}); this remains still unresolved
in positive characteristic as the convergence of the integral involved is not known in that case.

\begin{prop}
\label{prop:EP}
Let $\pi$ and $\pi'$ be finite-length, smooth
representations of a
reductive
$p$-adic
group $G$.
Then:
\begin{enumerate}[(a)]
\item
\label{item:EP-bilinear}
$\EP$ is a symmetric, $\Z$-bilinear form
on the Grothendieck group of finite-length
representations of $G$.
\item
\label{item:EP-loc-const}
$\EP$ is locally constant.
(A family $\{\pi_\lambda\}$ of representations on a fixed vector space $V$ 
is said to \emph{vary continuously}
if all $\pi_\lambda|_K$ are all equivalent for some compact open subgroup $K$,
and the matrix coefficients
$\langle \pi_\lambda v, \tilde v\rangle$
vary continuously in $\lambda$.)

\item
\label{item:EP-ps}
$\EP(\pi,\pi') = 0$
if $\pi$ or $\pi'$ is
induced from any proper parabolic subgroup in $G$.

\item
\label{item:EP-chars}
$\EP(\pi,\pi') = 
\int_{C_{\ellip}} \Theta(c)\bar{\Theta}'(c)\, dc$,
where $\Theta$ and $\Theta'$ are the characters of $\pi$ and $\pi'$ assumed to have the same
unitary central character, 
and $dc$ is a natural measure on the set
$C_{\ellip}$ of regular elliptic conjugacy classes in $G/Z(G)$. 
\end{enumerate}
\end{prop}

The Euler-Poincare pairing becomes especially useful because of the following two results,
concerning vanishing of higher $\Ext$ groups and Frobenius reciprocity for $\Ext$.

\begin{prop}
\label{prop:ext-vanish}
Suppose that $V$ in $\RR^\chi(G)$ has finite length, and that all of its
irreducible subquotients are subquotients of representations induced from supercuspidal representations
of a Levi factor of the standard parabolic subgroup $P$ of $G$,
defined by a subset $\Theta$ of the set of simple roots.
Then
$\Ext^i_{\RR^\chi(G)}(V,V') = 0$
for $i> d - |\Theta|$
and any representation $V'$ in $\RR^\chi(G)$ where $d$ is the $k$-split rank of $G/Z(G)$.
\end{prop}

\begin{proof}
This is \cite{schneider-stuhler:sheaves}*{Corollary III.3.3}.
\end{proof}

\begin{prop}[Frobenius reciprocity]
\label{prop:frobenius}
Let $P$ be a parabolic subgroup of $G$ with Levi factorization $P=MN$.
Let $\pi$ be a smooth representation of $G$, and $\sigma$
a smooth representation of $M$.
Then
$$
\Ext_{\RR^\chi(G)}^i (\pi,i_P^G(\sigma)) \cong \Ext_{\RR^\chi(M)}^i(r_N (\pi),\sigma),
$$
where $\RR^\chi(M)$ is the category of smooth representations of $M$
on which the center
of $G$ (which is always contained in $M$) acts via $\chi$,
and $r_N$ denotes the Jacquet functor.
\end{prop}

\begin{proof}
This is \cite{casselman:frobenius}*{Theorem A.12}.
\end{proof}

\begin{prop}
\label{prop:inequivalent}
Let $G$ be a reductive $p$-adic group over $k$, and $P$ a maximal $k$-parabolic
subgroup of $G$ with Levi decomposition $P=MN$.
Let $\sigma$ be an irreducible,
supercuspidal representation
of $M$,
and let
$\pi=i_{P}^{G}\sigma$. If $N_G(M)/M$ is nontrivial, it is of order $2$, 
in which case write $N_G(M)/M = \langle w  \rangle$. 
\begin{enumerate}

\item
If $N_G(M)/M$ is trivial, $\pi=i_{P}^{G}\sigma$ is irreducible. 

\item
If $N_G(M)/M = \langle w \rangle$, and 
$ \sigma \not \cong \sigma^w$, then if $\pi$ is reducible, it is indecomposable 
with distinct Jordan-H\"older factors.

\item
If $ \sigma \cong \sigma^w$, then by twisting $\pi$ by a character of $G$, 
we can assume $\sigma$ to be unitary, hence if $\pi$ is reducible, it is completely reducible, and is
a direct sum of two distinct irreducible subrepresentations.
\end{enumerate}
\end{prop}

\begin{proof}
Part (1) of the proposition is
\cite{casselman:book}*{Theorem 7.1.4},
and is nontrivial; the other parts 
are more elementary, and follow from considerations of the Jacquet module which we undertake now. In these
parts we do not have to go into the deeper aspects of the subject regarding when reducibility actually 
occurs.

For $P=MN$, let $P^-=M N^-$ be the opposite parabolic. Then $P^-$ and $P$ are conjugate in $G$ if and only if 
$N_G(M)\not = M$. If $P$ and $P^-$ are conjugate in $G$, then $P$ is the unique parabolic in $G$ up to conjugacy 
in its associate class;
otherwise, there are two distinct conjugacy classes of parabolics in the associate class of $P$. 
It follows from the {\it geometric lemma} that $r_N(\pi) = \sigma$ if $N_G(M)=M$, and
that if $N_G(M) \not = M$,
then $r_N(\pi)$ has Jordan-H\"older
factors $\sigma$ and $\sigma^w$. If $\sigma \not \cong \sigma^w$, then since $\sigma$ is supercuspidal, 
$r_N(\pi) = \sigma \oplus \sigma^w$. In this case if $\pi$ is reducible, with $\pi_1$ and $\pi_2$ as the 
Jordan-H\"older factors of $\pi$, then we can assume that $r_N(\pi_1) = \sigma$, and $r_N(\pi_2) = \sigma^w$.
From Frobenius reciprocity,
$$
\Hom_G[\pi_2,\pi] = \Hom_M[r_N(\pi_2), \sigma] = \Hom_M[\sigma^w, \sigma]= 0.
$$
proving that if $\sigma \not \cong \sigma^w$, 
and $\pi$ is reducible, it is indecomposable with distinct 
Jordan-H\"older factors, proving part (2) of the proposition.

Note that if $N_G(M)/M$ is nontrivial and $\sigma^w \cong \sigma$, 
$\sigma$ must be unitary when restricted to the intersection of $M$ and the derived group $[G,G]$ of $G$.
If the supercuspidal representation $\sigma$ of the Levi subgroup $M$ is unitary, then $\pi$ is 
completely reducible, and we see that the Jordan-H\"older factors of $\pi$ are distinct by a 
calculation of $\Hom_G[\pi,\pi] =  \Hom_M[r_N (\pi),\sigma]$, which is a two dimensional vector space over $\C$.
If $\sigma$ is not unitary when restricted to $M \cap [G,G]$, in particular
$\sigma \not \cong \sigma^w$, we see that the Jordan-H\"older factors of $\pi$ are distinct by
noting that their Jacquet modules are $\sigma$ and $\sigma^w$, proving 
part (3) of the proposition. 
\end{proof}

\section{Proof of Theorem \ref{thm:main}}
\begin{proof}
If $\pi$ is irreducible, then the result follows from
Proposition \ref{prop:EP}(\ref{item:EP-ps}),
together with Proposition \ref{prop:ext-vanish}.

Suppose from now on that $\pi$ is reducible. Assume first that we have a non-split short exact sequence
\begin{equation*}
\tag{$*$}
0
\longrightarrow \pi_1 
\longrightarrow \pi 
\longrightarrow \pi_2 
\longrightarrow 0.
\end{equation*}
From ($*$), we have that $\Ext^1(\pi_2,\pi_1)$ is nontrivial, and this by 
Proposition \ref{prop:inequivalent} implies that the inducing representation $\sigma$ is not 
unitary even after twisting by characters of $G$ (restricted to the Levi subgroup).
By replacing the inducing representation $\sigma$ with its Weyl conjugate,
we obtain another principal series representation $\pi'$
which will have $\pi_1$ as a 
quotient, and $\pi_2$ as a subrepresentation.
Since $\sigma$ is not unitary,
Proposition \ref{prop:inequivalent} implies that
$\pi'$ does not split.
Thus,
$\Ext^1(\pi_1,\pi_2)$ is nontrivial.

Working in the category $\RR^\chi(G)$,
apply $\Hom( \pi_1,-)$
to ($*$)
and consider the induced long exact sequence
$$
\begin{array}{llll}
0
&\longrightarrow \Hom(\pi_1 , \pi_1)
&\longrightarrow \Hom(\pi_1 , \pi)
&\longrightarrow \Hom(\pi_1 , \pi_2) \\
&\longrightarrow \Ext^1(\pi_1 , \pi_1)
&\longrightarrow \Ext^1(\pi_1 , \pi)
&\longrightarrow \Ext^1(\pi_1 , \pi_2) \\
&\longrightarrow \Ext^2(\pi_1 , \pi_1)
&\longrightarrow \cdots
\end{array}
$$
By Proposition \ref{prop:ext-vanish},
$\Ext^2(\pi_1,\pi_1) = 0$.
From Proposition \ref{prop:inequivalent},
$\Hom(\pi_1,\pi_2) = 0$,
so we have a short exact sequence
\begin{equation*}
\tag{$\bigtriangleup$}
0
\longrightarrow \Ext^1(\pi_1 , \pi_1)
\longrightarrow \Ext^1(\pi_1 , \pi)
\longrightarrow \Ext^1(\pi_1 , \pi_2) 
\longrightarrow 0.
\end{equation*}
Since
$\Ext^1(\pi_1 , \pi_2)$ is nonzero,
and
since $r_{ N} \pi_1 \cong \sigma$,
Frobenius reciprocity
(Proposition \ref{prop:frobenius})
gives
$$
\Ext^1_{\RR^\chi(G)}(\pi_1,\pi) \cong \Ext^1_{\RR^\chi(M)}( r_{N} \pi_1, \sigma)
\cong \Ext^1_{\RR^\chi(M)} (\sigma,\sigma).
$$
Let $\chi_\sigma$ denote the central character of $\sigma$.
Then $\sigma$ is projective in $\RR^{\chi_\sigma}(M)$.
Since $Z(M)/Z(G)$ has split rank $1$,
$\dim \Ext_{\RR^\chi(M)}^1(\sigma,\sigma) = 1$.
(This amounts to the assertion that $\Ext^1_{k^\times}(\mu,\mu) = \C$,
where $\mu$ is a one-dimensional character of $k^\times$.)
From ($\bigtriangleup$), we thus have that
$\Ext^1(\pi_1,\pi_1) = 0$
and
$\Ext^1(\pi_1,\pi_2) = \C$,
as desired.

We now turn to the case when $\pi = \pi_1 +\pi_2$.
This is the nontrivial part of the proposition,
where one wants to construct a nontrivial extension between $\pi_1$ and $\pi_2$,
even though the extension
afforded by the principal series representation in which they sit is split.

From Proposition \ref{prop:ext-vanish},
$\Ext^i(\pi_1,\pi_1) = 0$ for $i > 1$. 
From 
\cite{arthur:elliptic-tempered}*{Proposition 2.1(c)},
the character of $\pi_1$ does not vanish on the elliptic set.
By Proposition \ref{prop:EP}(\ref{item:EP-chars}),
$\EP(\pi_1,\pi_1)$ is positive.
Thus,
\begin{equation*}
\begin{split}
\dim\Ext^1(\pi_1,\pi_1) = \dim\Hom(\pi_1,\pi_1) - \EP(\pi_1,\pi_1) \\
= 1 - \EP(\pi_1,\pi_1) < 1,
\end{split}
\end{equation*}
and thus $\Ext^1(\pi_1,\pi_1) = 0$.
From Proposition \ref{prop:EP}(\ref{item:EP-ps}),
$\dim \Ext^1(\pi_1, \pi)=1$, so it follows that 
$\dim\Ext^1(\pi_1,\pi_2) = 1$.
The rest of the proposition follows
by symmetry between $\pi_1$ and $\pi_2$. 
\end{proof}

\begin{rem}
It may be worth emphasizing that although the proof of Theorem \ref{thm:main}
above might look straightforward, it 
uses rather deep Proposition \ref{prop:EP}(\ref{item:EP-chars}) which is known only in characteristic 0,
and hence so also this theorem.
\end{rem}

\section{A construction of Savin}
\label{sec:savin}

If $\pi$ is a reducible  unitary principal series representation of $\SL_2(k)$
then it has two inequivalent, irreducible subquotients $\pi_1$ and $\pi_2$.
By Theorem 1 we know that
$$
\Ext^1_{\SL_2(k)}(\pi_1,\pi_2) = \C.
$$
G. Savin has offered a natural construction of such an extension, at least when $\pi$ arises from an
unramified quadratic character of $k^\times$.
This construction may be useful in many similar situations,
so we outline it, referring to \cite{savin}
for details.
We begin with some generality.

Let  $K$ be an open compact subgroup of a split reductive $p$-adic group $G$. 
Let $\Hecke=C_c(K\backslash G/K)$ be the Hecke algebra of 
$K$-bi-invariant compactly supported functions on $G$.  If $V$ is a smooth $G$-module,
then $V^K$ is a left $\Hecke$-module. 
It is a standard fact that if $V$ is an irreducible $G$-module, and if  $V^K$ is non-zero, the latter
is an irreducible $\Hecke$-module. 
Conversely, 
if $E$ is a left $\Hecke$-module then 
\[ 
I(E) := C_c(G/K)\otimes_\Hecke E
\] 
is a smooth $G$-module. 
As a right $\Hecke$-module,  $C_c(G/K)$ can be decomposed as
\[
C_c(G/K)=C_c(G/K)'\oplus \Hecke,
\]
where $C_c(G/K)'$ denotes the sum of all non-trivial left $K$-submodules of $C_c(G/K)$.
It follows that 
$I(E)^K\cong E$, as $\Hecke$-modules. Note that $I(E)^K$ generates the $G$-module $I(E)$. 
Let $U(E)\subseteq I(E)$ be the sum of all $G$-submodules of $I(E)$ 
intersecting $I(E)^K$ trivially.  Let $J(E)$ be the quotient $I(E)/U(E)$.  Then $J(E)$ is generated by 
$J(E)^K\cong E$, and any submodule of $J(E)$ contains non-zero $K$-fixed vectors.
Using this, the following proposition is proved. 

\begin{prop}  
Let $E$ be an irreducible $\Hecke$-module.  Then $J(E)$ is the unique irreducible quotient of $I(E)$. 
\end{prop}

Assume now that $K$ is hyperspecial and let
$\I\subseteq K$ be an  Iwahori subgroup. Since $\Hecke$ is commutative, 
every irreducible $\Hecke$-module is one dimensional.
Pick one, and call it $\C_\chi$.
Any subquotient of $I(\C_\chi)$ is generated by its 
$\I$-fixed vectors.
As in
\cite{savin},
denoting by $X$ the cocharacter group of a maximal split torus of $G$, we have 
\[ 
I(\C_\chi)^\I=C_c(\I\backslash G/K) \otimes_\Hecke \C_\chi \cong \C[X]\otimes_{\C[X]^W} \C_\chi. 
\] 

From generality about integral extensions of commutative integrally closed domains,
$\C[X]\otimes_{\C[X]^W} \C_\chi$ has dimension 
equal to $|W|$, hence
$\dim(I(\C_\chi)^\I)=|W|$. 
We specialize further to $G=\SL_2(k)$.  Let $V=\pi_1$ be the unique irreducible tempered representation of $G$ such that 
$\dim(V^K)=\dim(V^\I)=1$. Then $I(V^K)$ has length 2, and is the representation of $\SL_2(k)$ corresponding to a non-trivial element of
$\Ext^1_{\SL_2(k)}(\pi_1,\pi_2) = \C$ that we desired to construct since the 
unique irreducible quotient of $I(V^K)$
is $V = \pi_1$. If  $U$ is the unique irreducible 
submodule of $I(V^K)$,
then $\dim(U^K)=0$ and $\dim(U^\I)=1$,
thus $U$ is the irreducible representation such that $V\oplus U$ is isomorphic to the representation induced from 
the unique non-trivial, unramified, quadratic character of $k^{\times}$, forcing $U$ to be $\pi_2$.

\section{Preliminary results for Theorem \ref{thm:sln}}
\label{sec:sln-prelim}

We recall a small part of the theory of types \cite{bushnell-kutzko:smooth}. 
The starting point is the 
fundamental result, due to Bernstein, that the category $\RR(G)$ of smooth complex 
representations of $G$ decomposes as a direct sum of certain indecomposable full subcategories,
now often called the Bernstein components of $\RR(G)$:
\[
\RR(G) = \coprod_{\fraks \in \mathcal{B}(G)} \RR_{\fraks}(G).
\] 
The indexing set $\mathcal{B}(G)$ consists of (equivalence classes of)
irreducible supercuspidal representations of Levi subgroups $M$ of $G$ 
up to conjugation by $G$ and twisting by unramified characters of $M$,
i.e., characters that are trivial on all compact subgroups of $M$.

Suppose $\fraks\in \mathcal{B}(G)$ corresponds to an irreducible supercuspidal representation,
say $\sigma$, of a
Levi subgroup $M$ of $G$.
The irreducible objects in $\RR_{\fraks}(G)$
are then precisely the irreducible subquotients of the various parabolically
induced representations $i_P^G(\sigma \nu)$
as $\nu$ varies through the unramified characters of $M$ and where $P$ is any parabolic subgroup 
of $G$ with Levi component $M$.

We note the following lemma which is a simple consequence of Bernstein theory but which, however, 
does not follow from Frobenius reciprocity.
The result can also be found in
\cite{vigneras:extensions}*{Theorem 6.1}.

\begin{lemma}
\label{lem:cuspidal-support}
Let $\pi_1$ and $\pi_2$ be two irreducible admissible representations of $G$ with different
cuspidal support. Then
$$
\Ext^i(\pi_1,\pi_2) = 0 \quad\text{for all}\quad i \geq 0.
$$
\end{lemma}
\begin{proof} If $\pi_1$ and $\pi_2$ belong to different Bernstein components, then there is nothing to
prove. If they belong to the same Bernstein component, then associated to the component is
an irreducible affine algebraic variety over $\C$ whose space of regular functions is the 
center of the corresponding category.  Now given two distinct points on the affine algebraic variety 
corresponding to $\pi_1$ and $\pi_2$,
there is an element, call it $f$, in the center of the category such that $f$ acts by 0 on
$\pi_1$, and by 1 on $\pi_2$. Standard homological algebra then proves that
$\Ext^i(\pi_1,\pi_2) = 0$ for all $i \geq 0$.
\end{proof}

A pair $(K,\rho)$
consisting of a compact open subgroup $K$ of $G$
and a smooth irreducible representation $\rho$ of $K$ is called an 
$\fraks$-\emph{type} if the irreducible smooth representations of $G$
that contain $\rho$ on restriction to $K$ are exactly the irreducible 
objects in $\RR_{\fraks}(G)$.
In this case, the category $\RR_{\fraks}(G)$ is equivalent to the category of
(left) modules over the intertwining 
or Hecke algebra of $\rho$.
More precisely, let $W$ denote the space of $\rho$ and write ${\Hecke}(G,\rho)$ for the space of compactly 
supported functions $\Phi:G \to \End(W^\vee)$ such that 
$$
\Phi(k_1gk_2) = \rho^\vee(k_1) \Phi(g) \rho^\vee(k_2),
$$
where, as usual,
$\rho^\vee$ is the dual of $\rho$.
This is a convolution algebra (with respect to a fixed Haar measure on $G$). 
The endomorphism algebra $\End_G(\ind_K^G \rho)$ is isomorphic to 
the opposite of the algebra ${\Hecke}(G,\rho)$, so that a right 
$\End_G(\ind_K^G \rho)$-module is naturally a 
left ${\Hecke}(G,\rho)$-module.  This allows one to give a natural 
${\Hecke}(G,\rho)$-module structure on $\Hom_K(\rho, \pi)$ for any
smooth representation $\pi$ of $G$.

We mention two basic examples of $\fraks$-types which served as precursors to the general theory.
In the first, $M = G$. Thus $\sigma$ is a supercuspidal representation of $G$ and the 
irreducible objects in $\RR_\fraks(G)$ 
are simply the unramified twists of $\sigma$. In this case, elementary
arguments show that the existence of an $\fraks$-type is closely
related to the statement that $\sigma$ is induced from a compact mod center subgroup of $G$
(see \cite{bushnell-kutzko:smooth}*{\S5.4}).
In particular, 
if $\sigma$ is induced in this way,
then an $\fraks$-type exists and is easily described in terms of the inducing data for $\sigma$.
We note that through the work of 
J.-K.\ Yu \cite{yu:supercuspidal}, Julee Kim \cite{jkim:exhaustion},
and S. Stevens \cite{stevens:classical-sc},
the existence of such types is now known for all reductive groups under a 
tameness hypothesis and for many classical groups in 
odd residual characteristic.
Types exist for $\GL(n)$ and $\SL(n)$ without
any restriction on residue characteristic by the work of Bushnell-Kutzko \cite{bushnell-kutzko:smooth}, 
and Goldberg-Roche \cites{goldberg-roche:sln-types,goldberg-roche:sln-hecke}.

In the second example, $\RR_\fraks(G)$ is defined by 
$\sigma$, the trivial representation of a minimal Levi subgroup $M$ of $G$. Since a minimal Levi subgroup
has no proper parabolic subgroup, the trivial representation of $M$ is supercuspidal; further, 
it is known that $M$ is compact modulo its center. 
In this case, the trivial representation of 
an Iwahori subgroup $\I$ provides an $\fraks$-type:
this is the classical result of Borel and Casselman that an irreducible 
smooth representation of $G$ contains 
non-trivial $\I$-fixed vectors if and only if it is a constituent of an unramified principal series. 
The general theory posits that these two examples are extreme instances of a general phenomenon.

A fundamental feature of Bushnell-Kutzko's theory is that parabolic induction can be transferred effectively
to the Hecke algebra setting and we make essential use of this feature below. 
We recall a special case which is more than adequate to our needs. 
Let $\sigma$ be an irreducible supercuspidal representation of a Levi subgroup $M$ of $G$ and write 
$\RR_{\fraks_M}(M)$ for the resulting component of $\RR(M)$. 
Thus the irreducible objects in $\RR_{\fraks_M}(M)$ are simply the various unramified twists of $\sigma$.
We also write
$\RR_{\fraks}(G)$ for the resulting component of $\RR(G)$.
We assume that $\RR_{\fraks_M}(M)$ admits a type $(K_M, \rho_M)$.
We assume also that $(K_M,\rho_M)$ admits a $G$-\emph{cover} $(K,\rho)$
whose definition due to Bushnell and Kutzko we recall below 
(see \cite{bushnell-kutzko:smooth} \S 8).

Given a parabolic $P=MN$, with opposite parabolic $P^- = MN^-$, we call a pair $(J,\tau)$ consisting of
a compact open subgroup $J$ of $G$, and a finite-dimensional irreducible representation $\tau$ of $J$
\emph{decomposed} with respect to $(P,M)$ if
\begin{enumerate}
\item $J = (J \cap N^-)\cdot (J\cap M) \cdot (J\cap N)$.
\item The groups $J\cap N^-$ and $J \cap N$ act trivially under $\tau$, so $\tau$ restricted
to $J_M= J \cap M$ is an irreducible representation; call it $\tau_M$.
\end{enumerate}

Let $I_G(\tau)$ denote the set of elements $g$ in $G$ such that there is a function $f$ in
${\Hecke}(G,\tau)$ whose support contains $g$. It can be seen that if $(J,\tau)$ is decomposed with respect to 
$(P,M)$, then
$$
I_M(\tau_M) = I_G(\tau) \cap M.
$$
Further, if $\phi \in {\Hecke}(M,\tau_M)$ has support $J_MzJ_M$ for some $z \in M$, there is a 
unique $T(\phi) = \Phi \in {\Hecke}(G,\tau)$ with support contained in $JzJ$, and with $\Phi(z) = \phi(z)$.
The map $T \colon \phi \rightarrow \Phi$ from ${\Hecke}(M,\tau_M)$ to ${\Hecke}(G,\tau)$ is an
isomorphism of vector spaces onto ${\Hecke}(G,\tau)_M$, the subspace of ${\Hecke}(G,\tau)$ with support
contained in $JMJ$.

One calls an element $z\in M$ \emph{positive} with respect to $(J, N)$ if it satisfies,
$$
z(J\cap N)z^{-1} \subset (J \cap N), \,\,\, z^{-1} (J\cap N^-)z \subset (J \cap N^-).
$$

Let $I^+$ denote the set of positive elements of $I_M(\tau_M)= I_G(\tau) \cap M$, and let
${\Hecke}(M,\tau_M)^+$ 
denote the space of functions in ${\Hecke}(M,\tau_M)$ with support
contained in $J_MI^+J_M$. Then the map $T$ 
from ${\Hecke}(M,\tau_M)$ to ${\Hecke}(G,\tau)$ when
restricted to ${\Hecke}(M,\tau_M)^+$ is an 
algebra homomorphism sending identity element of ${\Hecke}(M,\tau_M)$ to the identity
element of ${\Hecke}(G,\tau)$; it extends uniquely to an injective algebra homomorphism
from ${\Hecke}(M,\tau_M)$ to ${\Hecke}(G,\tau)$ when the pair $(J,\tau)$ is a $G$-cover
(to be defined below)
of $(J_M,\tau_M)$.

Define an element $\zeta$ of the center $Z(M)$ of $M$ to be \emph{strongly positive} if
it is positive, and has the property that given compact open subgroups
$H_1$ and $H_2$ of $N$, there is a power $\zeta^m$, $m \geq 0$, which conjugates $H_1$ 
into $H_2$, and similarly a property for subgroups of $N^-$ by negative powers of 
$\zeta$.

Here then is the definition of a $G$-cover.

\begin{defn}
Let $M$ be a Levi subgroup of a reductive group $G$. Let $J_M$ be a compact open subgroup of $M$,
and $(\tau_M,W)$ an irreducible smooth representation of $J_M$. Let $J$ be a compact open subgroup of $G$,
and $\tau$ an irreducible smooth representation of $J$. The pair $(J,\tau)$ is $G$-cover of $(J_M,\tau_M)$
if the following holds:

\begin{enumerate}
\item The pair $(J,\tau)$ is decomposed with respect to $(M,P)$, in the sense defined earlier,
for all parabolics $P$ with Levi $M$.

\item $J \cap M = J_M, $ and $\tau|_{J_M}= \tau_M$.

\item For every parabolic $P=MN$ with Levi $M$, there exists an invertible element 
of ${\Hecke}(G,\tau)$ supported on a double coset $J\zeta_PJ$ 
where $\zeta_P \in Z(M)$ is strongly $(J, N)$-positive.

\end{enumerate}
\end{defn}

The definition of a $G$-cover is tailored to achieve the following result,
which can be found in \cite{bushnell-kutzko:smooth}.

\begin{prop}
Let $P$ be a parabolic subgroup of a reductive $k$-group $G$,
and let $M$ be a Levi factor of $P$.
Let $J_M$ be a compact open subgroup of $M$,
and $(\tau_M,W)$ an irreducible smooth representation of $J_M$.
Suppose that $\RR_{\fraks_M}(M)$ 
is a component of $\RR(M)$, defined by the type 
$(\tau_M,W)$. Let $J$ be a compact open subgroup of $G$,
and $\tau$ an irreducible smooth representation of $J$.
If the 
pair $(J,\tau)$ is a $G$-cover of $(J_M,\tau_M)$,
then parabolic induction from $P$ to $G$
of representations in $\RR_{\fraks_M}(M)$ defines a component in $\RR(G)$ 
with $(J,\tau)$ as a type.
\end{prop}

Recall the following result of
Moy-Prasad, \cite{moy-prasad}*{Proposition 6.4}.
Let $\PP$ be a parahoric subgroup of a reductive group $G$ over $k$, with
$\PP^+$ the pro-unipotent radical of $\PP$.
If $\F_q$ is the residue field of $k$, then
$\PP/\PP^+$ is the group of rational points of a reductive $\F_q$-group.
There is a unique $\PP$-conjugacy class of Levi subgroups $M$ in $G$ such that $\M = \PP \cap M$
is a maximal parahoric subgroup in $M$ with
$$
\M/\M^+ \cong \PP/\PP^+.
$$
The following result of Morris \cite{morris:G-types}
constructs $G$-covers for all depth-zero types of Levi subgroups.
The relevance of this result for us is that in the tame case, i.e., $(n,p)=1$,
the representations of $\SL_n(k)$ that we consider have depth zero.
Although we will obtain $G$-covers for them from the work of Goldberg-Roche,
in the tame case we could have used Morris's result instead.
In fact,
Morris goes further to identify the Hecke algebra $\Hecke(G,\rho)$ too, but we do not go into that.

\begin{prop}
Let $G$ be a reductive algebraic group over a non-archimedean local field $k$.
Let $\PP$ be a parahoric subgroup of $G$, defining a Levi subgroup $M$, and maximal
parahoric $\M$ in $M$ as above with
$$
\M/\M^+ \cong \PP/\PP^+,
$$
allowing one to construct representations of
$\PP$ from representations of $\M/\M^+$.
Let $\rho$ be any irreducible 
representation of $\PP$ arising out of this construction.
Then $(\PP,\rho)$ is a 
$G$-cover of $(\M, \rho|_{\M})$.
\end{prop}

Let $P$ be a parabolic subgroup of $G$ with Levi component $M$.
The functor $i_P^G$ of normalized parabolic induction from
$\RR(M)$ to $\RR(G)$ takes $\RR_{\fraks_M}(M)$ to $\RR_{\fraks}(G)$. 
It therefore corresponds, under the equivalence of
$\RR_{\fraks}(G)$ with ${\Hecke}(G,\rho)$-modules
and its analogue for $M$, to a certain functor from 
${\Hecke}(M,\rho_M)\Mod$ to ${\Hecke}(G,\rho)\Mod$.
To describe this, we note that there is a certain (explicit) embedding of $\C$-algebras
$$
\lambda_P\colon {\Hecke}(M,\rho_M) \longrightarrow  {\Hecke}(G,\rho). 
$$
This induces a functor $(\lambda_P)\ast$
from ${\Hecke}(M,\rho_M)\Mod$ to ${\Hecke}(G,\rho)\Mod$,
given on objects by 
$$
S \longmapsto {\Hom}_{{\Hecke}(M,\rho_M)}({\Hecke}(G,\rho), S)
$$
where ${\Hecke}(G,\rho)$ is viewed as a left ${\Hecke}(M,\rho_M)$-module via $\lambda_P$ and
${\Hecke}(G,\rho)$ acts by right translations. 
We have the following commutative diagram (up to natural equivalence) by
\cite{bushnell-kutzko:smooth}*{Cor.~8.4}:
\begin{equation}  \label{ind-types}
\begin{CD}
\RR_{\fraks}(G) @>\simeq>>    {\Hecke}(G,\rho)\Mod  \\
@A{i_P^G}AA                               @AA{(\lambda_P) \ast}A  \\
\RR_{\fraks_M}(M) @>\simeq>> {\Hecke}(M,\rho_M)\Mod. \\
\end{CD}
\end{equation}
In other words, normalized parabolic induction from $\RR_{\fraks_M}(M)$ to $\RR_{\fraks}(G)$ corresponds to 
$(\lambda_P) \ast$ under the equivalences of the theory of types. 
(Note although \cite{bushnell-kutzko:smooth} explicitly treats only unnormalized induction, 
it is a trivial matter to adjust the arguments so that they apply to normalized induction.)

\section{Proof of Theorem \ref{thm:sln} in the totally ramified case}
\label{sec:tot-ram}
We set $G = SL_n(k)$.
Let $T$ denote the standard split torus of diagonal elements in $G$ and 
$T^1$ the unique maximal compact subgroup of $T$.
We write 
\[
A = \set{(a_1, \cdots, a_n) \in \Z^n }{\sum a_i = 0}.
\]
Fix a uniformizer $\varpi$ in $k$.
Consider the map $a \mapsto \varpi^a \colon A \to T$ where 
\[
\varpi^a = \diag(\varpi^{a_1}, \ldots, \varpi^{a_n}), \quad\text{for}\quad a = (a_1, \ldots, a_n).
\] 
This splits the inclusion $T^1 \hookrightarrow T$, i.e.,
\begin{equation} \label{splitting}
(t_1, a) \mapsto t_1 \varpi^a \colon T^1 \times A \overset{\simeq}\longrightarrow  T
\end{equation}
is an isomorphism.
In this way, we can view characters of $T$ as pairs consisting of characters of $T^1$ and characters of $A$
(equivalently, unramified characters of $T$). 

Let $\omega \colon k^\times \rightarrow \C^\times$ be a character of order $n$ such that the restriction
of $\omega$ to $\OO^\times$, call it $\omega_\OO$, remains of order $n$, where 
$\OO$ denotes the ring of integers in $k$. Let $\chi$ be the character of $T^1$ given by 
\[
\chi(\diag(x_1, x_2, \ldots, x_n))
= \omega_\OO (x_1)  \omega_\OO (x_2)^2 \cdots  \omega_\OO (x_n)^n. 
\]
We are interested in the resulting Bernstein component $\RR_\chi(G)$.
The irreducible objects in this component consist of the irreducible 
subquotients of the family of induced representations $i_B^G (\chi \nu)$
as $\nu$ varies through the unramified characters of $T$
(and $B$ is any Borel subgroup containing $T$). 

We write $\RR_\chi(T)$ for the Bernstein component of $T$ determined by $\chi$.
The irreducible objects in $\RR_\chi(T)$ are simply the various
extensions of $\chi$ to $T$.
It is obvious that $(T^1, \chi)$ is a type for $\RR_\chi(T)$.
By \cite{goldberg-roche:sln-types}, there is a $G$-cover
$(K,\rho)$ of $(T^1, \chi)$ which is therefore a type for $\RR_\chi(G)$. 
If $\omega$ is trivial on $1+\varpi \OO$, then 
$K$ is the Iwahori subgroup of $\SL_n(k)$. If $\omega$ is trivial on $1+\varpi^n\OO$ but not on 
$1+\varpi^{n-1} \OO$, for $n > 1$, then
$$
K = N^-([(n+1)/2]) \cdot T(\OO) \cdot N([n/2]),
$$
where $[x]$ denotes the integral part of the rational number $x$; and $N(i)$
(resp.\ $N^-(i)$) denotes the group of upper- (resp.\ lower-) triangular
unipotent matrices with non-diagonal entries in $\varpi^i\OO$.
The restriction of $\rho$ 
to $T(\OO)$ is the character $(1,\omega,\cdots, \omega^{n-1})$.

We next describe the Hecke algebra ${\Hecke}(G,\rho)$ 
and the algebra embedding $\lambda_B:{\Hecke}(T, \chi) \to {\Hecke}(G,\rho)$
(for $B$ a fixed Borel containing $T$).

To simplify some formulas, we take convolution in ${\Hecke}(T,\chi)$
(resp.\ ${\Hecke}(G,\rho)$) with respect to the Haar measure that gives
$T^1$ (resp. $K$) unit measure. 
For $a \in A$, let $\phi_a$ denote the 
the unique function in ${\Hecke}(T,\chi)$
with support $T^1 \varpi^a$ such that $\phi_a(\varpi^a) = 1$.
The assignment $\phi_a \mapsto a$, for 
$a \in A$, clearly extends to a $\C$-algebra isomorphism ${\Hecke}(T,\chi) \simeq \C[A]$. 

\begin{prop}
\label{prop:sln-hecke}
Let $H = A \rtimes \Z/n$ where $\Z/n$ acts on $A$ by cyclic permutation of the co-ordinates.
Then there is a $\C$-algebra isomorphism
${\Hecke}(G,\rho)  \simeq    \C[H]$, the complex group algebra of $H$.  This fits into a commutative diagram 
\begin{equation}  \label{algs-sl}
\begin{CD}
{\Hecke}(G,\rho)    @>\simeq>>    \C[H]  \\
@A{\lambda_B}AA        @AAA   \\
{\Hecke}(T,\chi)       @>\simeq>>   \C[A] \\
\end{CD}
\end{equation}
in which the right vertical arrow is the obvious inclusion
and the bottom horizontal arrow is the isomorphism that sends $\phi_a$ to $a$, for $a \in A$. 
\end{prop}

\begin{proof} 
Theorem~11.1 of  \cite{goldberg-roche:sln-hecke}  gives the following description of ${\Hecke}(G,\rho)$.   
First, for $a \in A$, we set $\Phi_a = \lambda_B(\phi_a)$ so that 
\[
\Phi_a \Phi_{a'} = \Phi_{a+a'},
\]
for all $a, a' \in A$. 
Writing $w$ for the cycle $(1 \, 2 \ldots n)$,
there is a special function $\Phi_w \in {\Hecke}(G,\rho)$ that satisfies
\begin{enumerate}
\item
$\Phi_w^n = \Phi_0$, the identity element of ${\Hecke}(G,\rho)$, 
\item
$\Phi_w \Phi_a \Phi_w^{-1} \doteq \Phi_{w(a)}$, for all $a \in A$.
\end{enumerate} 
Here $w$ acts on $A$ in the obvious way (by cyclic permutation of the co-ordinates),
and $\doteq$ denotes equality up to multiplication by scalars. 
(Note that it follows from (2) that the order of $\Phi_w$ is exactly $n$.) 
Finally, ${\Hecke}(G,\rho)$ is generated as a $\C$-algebra by $\Phi_w$ and the elements $\Phi_a$,
for $a \in A$. 

To prove the Proposition, we will show that (2) is actually an equality.
For this, we consider the induced representation $i_B^G(\chi)$ (viewing $\chi$ as
a character of $T$ that is trivial on $A$).
This decomposes as a sum of $n$ distinct irreducible subrepresentations. This can be seen by noting
the following:
\begin{itemize}
\item A unitary principal series representation of $\GL_n(k)$ is irreducible.
\item An irreducible admissible representation $\pi$ of $\GL_n(k)$ when restricted to $\SL_n(k)$ 
decomposes as a sum of a finite collection of irreducible representations whose cardinality is the same
as the cardinality of self-twists of $\pi$:
$$
\{ \alpha \colon k^\times \rightarrow \C^\times| \pi \otimes \alpha \cong \pi \}
= \{1,\omega,\cdots,\omega^{n-1}\}.
$$
\end{itemize}

We now appeal to the diagram (\ref{ind-types}).
The ${\Hecke}(G,\rho)$-module
that corresponds to $i_B^G(\chi)$ has dimension $n$ and so must
split as a sum of $n$ one-dimensional submodules. 
Note that each $\phi_a$, for $a \in A$,
acts trivially on the ${\Hecke}(T,\chi)$-module corresponding to the character $\chi$ of $T$.
It follows easily that there is a $\C$-algebra homomorphism $\Lambda:{\Hecke}(G,\rho) \to \C$ 
such that $\Lambda(\Phi_a) = 1$, for all $a \in A$ (in fact, there are $n$ such homomorphisms). 
Applying $\Lambda$ to (2), we see that (2) must be an equality.
\end{proof}

Combining (\ref{algs-sl})
(or more properly the diagram induced by (\ref{algs-sl}) on module categories) and (\ref{ind-types}),
we obtain a
commutative diagram of functors (up to equivalence)
\begin{equation} \label{key}
\begin{CD}
\RR_\chi(G)   @>\simeq>>    \C[H]\Mod  \\
@A{i_B^G}AA                               @AA{i}A  \\
\RR_\chi(T)    @>\simeq>>   \C[A]\Mod. \\
\end{CD}
\end{equation}
Explicitly,
if $M$ is a $\C[A]$-module, then $i(M) = \Hom_{\C[A]}(\C[H], M)$ where, as above,
the $\C[H]$-action is given by right translations. 
Let $\nu$ be an unramified character of $T$ viewed as a character of $A$ via $a \mapsto \nu(\varpi^a)$.
The bottom horizontal arrow takes the object $\chi \nu$ in $\RR_\chi(T)$
to the simple $\C[A]$-module $\C_\nu$ corresponding to $\nu$. 

We are interested in a particular family of induced representations in $\RR_\chi(G)$.
To describe this family, let $\omega$ be an $n$-th root of unity and write
$\nu_\omega$ for the unramified character of $T$ given by 
\[
\nu_\omega(\varpi^a) = \omega^{a_1} \omega^{2a_2}  \cdots \omega^{n a_n}, 
\]
for $a = \diag(a_1, \ldots, a_n)$.
To simplify the notation,
we write $\C_\omega$ in place of $\C_{\nu_\omega}$ for the $\C[A]$-module corresponding to $\nu_\omega$.
By (\ref{key}),
the induced representation $i_B^G(\chi \nu_\omega)$ corresponds to the $\C[H]$-module $i(\C_\omega)$.

Observe that $\C_\omega$ is fixed under the action of $\Z/n$ on $A$ (by cyclic permutations). Indeed,
\begin{align*}
(k_n,k_1,\ldots, k_{n-1}) &\mapsto
\omega^{k_n} \omega^{2k_1} \cdots  \omega^{nk_{n-1}} \\
&= 
\omega^{k_1} \omega^{k_2} \cdots  \omega^{k_n}
\bigl(
\omega^{k_1} \omega^{2k_2} \cdots  \omega^{nk_n}
\bigr) \\
&=
\omega^{k_1} \omega^{2k_2} \cdots  \omega^{nk_n}.
\end{align*}
It follows that for any character $\eta:\Z/n \to \C^\times$
the $\C[A]$-module $\C_\omega$ extends to a $\C[H]$ module 
$\C_{\omega, \eta}$ in which $\Z/n$ acts by $\eta$ and
\[
i(\C_\omega) = \bigoplus_\eta \C_{\omega, \eta}, 
\]
as $\eta$ varies through the distinct characters of $\Z/n$.

To finish, we therefore need to determine $\Ext^i_{\C[H]}(\C_{\omega, \eta}, \C_{\omega, \eta'})$
for all characters $\eta, \eta'$ of $\Z/n$. We have 
\begin{align*}
\Ext^i_{\C[H]}(\C_{\omega, \eta}, \C_{\omega, \eta'})
&\simeq  \Ext^i_{\C[H]}(\C_{1,1}, \C_{1, \eta' \eta^{-1}}) \\
&= H^i(H, \C_{1, \eta' \eta^{-1}}).
\end{align*}
Of course,
$\C_{1, \eta' \eta^{-1}}$ is a character of $H$ that is trivial on $A$.
To compute these cohomology groups, we use the following 
general result.

\begin{lemma}
\label{lem:cohom-normal}
Let $N$ be a finite-index normal subgroup of a group $G$, and 
$V$ a $\C[G]$-module.
Then
$$
H^i(G,V) \cong H^i(N,V)^{G/N}.
$$
\end{lemma}
\begin{proof}
This follows from the spectral sequence which calculates cohomology of
$G$ in terms of that of $N$ after we have noted that since $G/N$ is finite,
it has no cohomology in positive degree for a coefficient system which is a $\C$-vector space.
\end{proof}

\begin{cor}
\label{lem:shapiro}
Let $N$ be a normal subgroup of a group $G$ of finite index. 
Let $\tau$ be a finite-dimensional complex 
representation of $G$ on which $N$ operates trivially.  Then,
$$
H^i(G,\tau)
\cong
\left [ H^i(N, \C)\otimes \tau \right ]^G.
$$
\end{cor}
\begin{proof} 
By the previous lemma,
\begin{equation*}
H^i(G,\tau)
\cong   H^i(N,\tau)^{G/N} 
\cong  [H^i(N, \C) \otimes \tau]^G 
.
\qedhere
\end{equation*}
\end{proof}

This corollary allows us to calculate $H^i(H,\C_{1,\eta})$
as follows.

\begin{cor}
For a character 
$\eta \colon \Z/n \longrightarrow \C^\times$,
$$
H^i(H,\C_{1,\eta}) =  \Lambda^i(A^\vee \otimes \C)[\eta],
$$
where $\Lambda^i(A^\vee \otimes \C)[\eta],$ is the $\eta$-isotypic component of 
$\Lambda^i(A^\vee \otimes \C),$ for the action of $\Z/n$ as cyclic 
permutations on $A$.
\end{cor}

\begin{proof}
From Corollary \ref{lem:shapiro},
we have that
$H^i(H,\C_{1, \eta}) = H^i(A,\C)[\eta^{-1}]$.
Since the cohomology of a free abelian group is the exterior algebra
on its dual,
the corollary follows.
\end{proof}

Theorem \ref{thm:sln} in the ramified case now follows from the fact that $A^\vee \otimes \C$ as 
a module for $\Z/n$ is the sum
of all nontrivial characters of $\Z/n$.

\section{Preliminaries on Iwahori-Hecke algebras}

Now suppose that $\omega$ is an unramified character of $k^\times$ of order $n$, and we are considering the 
principal series representation $\Ps(1,\omega,\cdots,\omega^{n-1})$ of $\GL_n(k)$ restricted to $\SL_n(k)$. 
In this case the corresponding
Hecke algebra governing the situation is the Iwahori-Hecke algebra, which we review below in greater 
generality than needed for the problem at hand.

Let $G$ be an unramified group, i.e., a quasi-split group over $k$ 
which splits over an unramified extension of $k$ with $\I$ as an Iwahori subgroup
of $G$, with $\I \subset K$,
a hyperspecial maximal compact subgroup of $G$.
Let $T$ be a maximal torus in $G$ which is maximally split, such that $T(\OO)\subset \I$. (Recall that
since $G$ is unramified, so is $T$, and hence it makes sense to speak of $T(\OO)$ which is the
maximal compact subgroup of $T$.)
Let $W = N(T)(k)/T(k)$ be the Weyl group associated to the torus $T$. Let $X_{*}(T)$ 
be the cocharacter group of $T$. Fix
a uniformizer $\varpi$ in $k$, and for a cocharacter $\mu$ of $T$, let $\varpi^\mu$ denote the image
of $\varpi$ in $T$ under the map $\mu \colon k^\times \rightarrow T$.
The map $\mu\rightarrow \varpi^\mu$ gives an
isomorphism of $X_{*}(T)$ 
with $T/T(\OO)$, and hence induces an isomorphism of 
the group ring $R=\Z[X_{*}(T)]$
with $\Hecke(T\doubleslash T(\OO))$.

We recall (from \cite{haines-kottwitz-aprasad:iwahori-hecke})
that according to the Bernstein presentation of the Iwahori-Hecke algebra
$\Hecke(\I) = \Hecke(G\doubleslash \I)$,
there is the subalgebra $\Hecke(T\doubleslash T(\OO))$ generated as a vector space 
by the elements
$\I \varpi^\mu \I $ for $\mu$ in the set of coweights; 
the multiplication is
$(\I \varpi^\mu \I) (\I \varpi^\nu \I) = \I \varpi^{\mu + \nu} \I$
for $\mu$ and $\nu$ dominant coweights. 
The algebra $\Hecke(T\doubleslash T(\OO))$ 
is a Laurent polynomial algebra $\Z[X_{*}(T)]$.
There is also the subalgebra
$\Hecke(K\doubleslash \I)$ of the Iwahori Hecke algebra consisting 
of $\I$-bi-invariant functions on $G$ with support in $K$.
The natural map
$$
\Hecke(K\doubleslash \I) \otimes \Hecke(T\doubleslash T(\OO)) \longrightarrow \Hecke(\I)
$$
is an isomorphism of vector spaces.
In particular, $\Hecke(\I)$ is a free
module over $R= \Hecke(T\doubleslash T(\OO))$ 
of rank equal to the order of $W$. Furthermore,
an irreducible representation of 
$\Hecke( \I)$, when restricted to the commutative 
subalgebra $\Hecke(T\doubleslash T(\OO))$, breaks up as a sum of characters 
of $\Hecke(T\doubleslash T(\OO))$, 
which are just unramified characters of $T$
which are all conjugate 
under the action of $W$.
Any character in this Weyl orbit of characters
of $T$
is an inducing character for the corresponding unramified principal series
representation of $G$ in which this representation of $\Hecke( \I)$
is contained; in particular, the unramified principal series representation $\Ps(1,\omega,\cdots, \omega^{n-1})$ 
defines a character $\nu_\omega$ of $R= \Hecke(T\doubleslash T(\OO))$.

It is known that
$R^W= \Hecke(T\doubleslash T(\OO))^W$ is the center of 
$\Hecke( \I)$, 
and that if $Q$ denotes the quotient field of $R$,
then the algebra $\Hecke(\I) \otimes _{R^W} Q^W = \Hecke(\I) \otimes _{R} Q$
is isomorphic to what is called a 
twisted group ring of $W$ over $Q$ with the natural action of $W$ on $R$ and hence on $Q$.
In fact, we do not need to invert all the nonzero
elements of $R$ to get to the twisted group ring, and inverting just one element,
$$
\delta = \prod(1-q^{-1} \varpi^{\alpha^\vee})
$$
(where $q$ is the order of the residue field of $k$,
and the product is over all coroots $\alpha^\vee$),
is sufficient.
Clearly $\delta$ is a
$W$-invariant element of $R$, so belongs to the center $Z = R^W$ of 
$\Hecke( \I)$. Note that $\delta$ is not invertible in $R$ as 
$R$ being a Laurent polynomial algebra,
the only invertible elements of $R$ are the monomials.

We now localize 
$\Hecke( \I)$, $R$, $Z$
at the central multiplicative set given by the powers of $\delta$. 
Write $\Hecke( \I)_\delta$, $R_\delta$, $Z_\delta$ for these localizations.
The algebra $\Hecke( \I)_\delta$ has a simple structure. In fact, 
$$
\Hecke( \I)_\delta = \oplus_{w \in W} R_\delta K_w,
$$
where the normalized intertwining operators $K_w$
are as described in
\cite{haines-kottwitz-aprasad:iwahori-hecke}*{\S2.2}.
Now $W$ acts naturally on $R$ and $R_\delta$ and we have
$$
K_w r = w(r) K_w \qquad \text{for all $r \in R$}.
$$
We also have
$$
K_wK_{w'} = K_{ww'};
$$
these equations determine the algebra structure on 
$\Hecke( \I)_\delta$, and prove that $\Hecke( \I)_\delta \cong R_\delta[W]$. 

Note that from the explicit form of $\delta$ given above, $\nu_\omega(\delta) \neq 0$, and hence 
the character $\nu_\omega$ of $R$ that we work with extends uniquely to $R_\delta$.
We continue to write $\nu_\omega$ for this extension to $R_\delta$.

\section{Proof of Theorem \ref{thm:sln} in the unramified case}

Let
$\omega \colon k^\times\longrightarrow \C^\times$ be an unramified 
character of order $n$.
Recall that the principal series representation
$\pi=\Ps(1,\omega,\cdots,\omega^{n-1})$ of $\GL_n(k)$ decomposes as a direct
sum $\pi = \sum_{\alpha} \pi_\alpha$
of $n$ irreducible admissible representations of $\SL_n(k)$ 
where $\alpha \in k^\times/\ker(\omega)$,
all of which have Iwahori-fixed vectors. Extensions between these
can therefore be determined through the Iwahori-Hecke algebra $\Hecke(\I)$ of $G$. 
Since the space of $\I$-invariants in a principal series representation of any split group,
in particular $\SL_n(k)$, has dimension equal to the order $|W|$ of the Weyl group $W$, 
the representations of the Iwahori-Hecke algebra corresponding
to any $\pi_\alpha$ is of dimension $(n-1)!$ (all being of equal dimension).
To justify this, 
we note that $\dim(\pi_\alpha^\I)$ is independent of $\alpha$ since:
\begin{enumerate}
\item $\GL_n(k)$ operates
transitively on the set of $\pi_\alpha$.

\item If $N(\I)$ denotes the normalizer of $\I$ in $\GL_n(k)$, then
$N(\I)\cdot \SL_n(k) = \GL_n(k)$
since $\I$ is normalized by an element of $\GL_n(k)$ whose determinant is a uniformizer of $k$.
For example, if $\I$ is the ``standard'' Iwahori subgroup, then one such element is
$$
\begin{pmatrix}
0 & 1  &  0 &0& \cdots& 0 \\
&  0 & 1 & \ddots  & \cdots & 0 \\
& &  0& \ddots & 0 &  0\\  
&&&\ddots & 1&  0 \\
&&&& 0& 1 \\
\varpi &&&&& 0
\end{pmatrix}.
$$
\end{enumerate}

Using the notation from the previous section, our context consists of the following 
chain of $\C$-algebras: $R^W \subset R \subset \Hecke(\I)$ where $R$ as well as $R^W$ are Laurent
polynomial algebras, and $R^W$ is the center of $\Hecke(\I)$ which we now abbreviate to $\Hecke$. 
We have two modules $M_1,M_2$ over $\Hecke$
which are of dimension
 $(n-1)!$ over $\C$ arising from two irreducible components of 
the principal series representation $\Ps(1,\omega,\omega^2,\cdots, \omega^{n-1})$ of $\GL_n(k)$ restricted to $\SL_n(k)$,
and we are interested in calculating:
$$
\Ext^i_\Hecke(M_1,M_2).
$$

From results of the previous section, we know that there is an element $\delta$ in $R^W$, such that
inclusion $R_\delta \subset \Hecke_\delta$ is the inclusion
$R_\delta \subset R_\delta[W]$. We also know from the previous section that the element $\delta$ acts
invertibly on $M_1$, and $M_2$, and therefore $M_1$ and $M_2$ can be considered as modules for
$\Hecke_\delta = R_\delta[W]$.
Since the inclusion $\Hecke \subset \Hecke_\delta$ is flat, generalities from homological algebra 
imply that:
$$
\Ext^i_{{\Hecke}}(M_1,M_2) \cong \Ext^i_{\Hecke_\delta}(M_1,M_2).
$$

Given the inclusion of the twisted group rings $R[W] \subset R_\delta[W]$, let 
$M_1'$, resp. $M_2'$, be the module $M_1$, resp. $M_2$, restricted to $R[W]$. Then we have,
$$
\Ext^i_{R[W]}(M_1',M_2') \cong \Ext^i_{R_\delta[W]}(M_1,M_2).
$$

The twisted group ring $R[W]$ is the group ring of $A \rtimes S_n$ where 
$$
A = \set{(k_1, \cdots, k_n) \in \Z^n }{\sum k_i = 0},
$$
on which there is the natural action of the symmetric group $S_n$. The modules
$M_1'$ and $M_2'$ can therefore be considered as irreducible representations, say
$M_1'',M_2''$ of $A\rtimes S_n$, and we have
$$
\Ext^i_{R[W]}(M_1',M_2')  \cong \Ext_{A\rtimes S_n}^i(M_1'',M_2'').
$$

Thus we are led to a question about extensions between representations of a group, which
in this case is $A\rtimes S_n$.
Such questions are well-known to be related to
cohomology of groups, using which
we will eventually be able to prove that
$$
\Ext_{A\rtimes S_n}^i(M_1'',M_2'')  \cong  \Ext_{A\rtimes \Z/n}^i(\chi_1,\chi_2)
$$
where $\Z/n$ is the cyclic group generated by the $n$-cycle $(1,2,\cdots, n)$ in $S_n$, and $\chi_1,\chi_2$
are characters of $A\rtimes \Z/n$ which extend the character
$\phi \colon (k_1,\ldots ,k_n) \mapsto \omega^{k_1} \omega^{2k_2} \cdots \omega^{nk_n}$
of $A$ with the property that,
\begin{eqnarray*}
M_1'' & = & \Ind_{A \rtimes \Z/n}^{A \rtimes S_n} \chi_1, \\
M_2'' & = & \Ind_{A \rtimes \Z/n}^{A \rtimes S_n} \chi_2.
\end{eqnarray*}

The existence of the characters $\chi_1,\chi_2$ with the above properties is a simple consequence
of Clifford theory since the character of $A$ being considered has stabilizer $\Z/n$ 
generated by the $n$-cycle $(1,2,\cdots,n)$ in $S_n$.

Thus our calculations made in \S\ref{sec:tot-ram} for the totally 
ramified case become available, proving Theorem \ref{thm:sln}.
To carry out this outline, we begin with some simple generalities.

\begin{lemma}
Let $G$ be a group, and $V$ a $\C[G]$-module.
Assume that there is 
an element $z$ of the center of $G$
which operates by a scalar $\lambda_z \not = 1$ on $V$. Then $H^i(G,V) =0$ 
for all $i \geq 0$.
\end{lemma}
\begin{proof} The proof of this well known lemma depends on the observation that there
is a natural action of $G$ on $H^i(G,V)$ in which $g \in G$ acts on $G$ by 
conjugation, and which on coefficients $V$ acts by $v\rightarrow g^{-1}v$.
This action of $G$ on $H^i(G,V)$ is known to be the identity,
cf.\ Proposition 3 on page 116 of
\cite{serre:local-fields}.
On the other hand, the
element $z$ in the center of $G$ operates on $H^i(G,V)$ by $\lambda_z^{-1} \not = 1$, proving the lemma.
\end{proof}

Using this, we have:

\begin{prop}
\label{prop:triv-cohomology}
Let $A$ be a finitely-generated free abelian group on which
$\Z/n$ operates. Let $H = A \rtimes \Z/n$. Then for an 
irreducible finite-dimensional
complex representation $V$ of $H$,
$H^i(H,V) = 0$ unless $A$ acts trivially on $V$.
\end{prop}

\begin{proof}
Note that by Clifford theory, the representation $V$ is
obtained as induction of a character $\chi$ 
of a subgroup $H'$ of $H$ containing $A$, i.e., $V = \Ind_{H'}^H \chi$.
By Shapiro's lemma, $H^i(H,V)= H^i(H',\chi)$. The proof is then clear by using the 
previous lemma (applied to $G=A$, an abelian group!) combined with
Lemma \ref{lem:cohom-normal}.
\end{proof}

We come now to the main proposition needed for our work.
Let
$$
A = \set{(k_1,\cdots, k_n) \in \Z^n }{\sum k_i = 0},
$$
on which there is the natural action of the symmetric group $S_n$ which contains
the $n$-cycle $(1,2,\cdots,n)$, so the group $\Z/n$ generated by this cycle too operates
on $A$. This allows one to construct groups $H=A \rtimes \Z/n$ and $\tilde{H}= A \rtimes S_n$.
Let 
$$
\phi \colon 
(k_1,\ldots ,k_n) \mapsto \omega^{k_1} \omega^{2k_2} \cdots \omega^{nk_n}.
$$
be the character of order $n$ of $A$ as before; as noted earlier, the character $\phi$ of $A$ 
is invariant under the cyclic permutation action of $\Z/n$ on $A$.

\begin{prop}
Let $\chi_1$ and $\chi_2$ be any two extensions of the character $\phi$ of 
$A$ to characters
of $H= A \rtimes \Z/n$. 
Call $M_1$, resp. $M_2$, the representation of $\tilde{H}= A \rtimes S_n$,
obtained by inducing the characters $\chi_1$, $\chi_2$ of $H$. Then 
$$
\Ext^i_{\tilde{H}}(M_1,M_2) \cong \Ext^i_H(\chi_1,\chi_2).
$$
\end{prop}

\begin{proof}
We recall the generality that
$$
\Ext^i_{\tilde{H}}(M_1,M_2) \cong H^i(\tilde{H}, M_1^\vee \otimes M_2).
$$
Since $M_j = \Ind_H^{\tilde{H}} \chi_j$ 
(for $j=1,2$),
we have
$$
M_1^\vee \otimes M_2 \cong \Ind_H^{\tilde{H}} (\chi_1^{-1} \otimes M_2|_H).
$$
By Shapiro's lemma it follows that
$$
H^i(\tilde{H}, M_1^\vee \otimes M_2) = H^i(H, \chi_1^{-1} \otimes M_2|_H).
$$
Since the stabilizer of the character $\phi$ of $A$ is the group 
$H = A \rtimes \Z/n$, the restriction of the representation $M_2$ to $A$ consists 
of all {\it distinct} conjugates of the character $\phi$ under the symmetric group $S_n$ (with
$\Z/n$ as the isotropy of $\phi$).

Thus the part of the representation  $\chi_1^{-1} \otimes M_2|_H$ of $H$ on which $A$ acts
trivially is nothing but the one-dimensional representation $\chi_1^{-1} \chi_2$ of $H$.
By Proposition \ref{prop:triv-cohomology}
it follows that
$$
H^i(H, \chi_1^{-1} \otimes M_2|_H) = H^i(H,\chi_1^{-1} \chi_2).
$$
Again noting the generality
$$
\Ext^i_{H}(\chi_1,\chi_2) \cong H^i({H}, \chi_1^{-1} \chi_2),
$$
the proposition is proved.
\end{proof}

\section{A question of compatibility}
Theorem \ref{thm:sln} has been stated after fixing an arbitrary base point, called $\pi_1$, among
the irreducible components of the principal series representation $\Ps(1,\omega,\cdots,\omega^{n-1})$,
which gives rise to a parametrization of all components as
$(\pi_1)^{\langle e \rangle} = \pi_e$ for $e \in k^\times/\ker(\omega)$
by inner-conjugation action of $k^\times$ on $\SL_n(k)$. On the
other hand, the Hecke algebras, eventually identified to the group algebra of
$A \rtimes \Z/n$ in the ramified case, and of $A \rtimes S_n$ in the unramified case, give rise to their
own parametrizations.
The question arises: how do we relate these two very different looking 
parametrizations?

Recall that a character of $A$ determines an unramified principal series representation of $\SL_n(k)$.
Each such character is contained in an irreducible representation
of $A \rtimes \Z/n$.
When the character of $A$ has $n$ distinct conjugates under the action 
of $\Z/n$, one constructs this latter representation via induction to $A \rtimes \Z/n$,
and there are no choices to be made: the
character of $A$ uniquely determines the irreducible representation of $A \rtimes \Z/n$ to
which it belongs. 
However, in our case the character of $A$ is invariant under the action of $\Z/n$, so it
extends in $n$ distinct ways to
$A \rtimes \Z/n$.
These extended characters of $A \rtimes \Z/n$ are permuted transitively by 
multiplication by characters of $\Z/n$ since $\Z/n$ is a quotient
of $A\rtimes \Z/n$.

The following proposition answers the question of compatibility. We let $G=\SL_n(k)$ below.

\begin{prop}
\label{prop:compatibility}
For $e \in k^\times /\ker(\omega)$, the map $\chi \mapsto \chi(e)$
establishes an identification of the character group
of $\{1,\omega,\cdots, \omega^{n-1}\} = \Z/n$ with $k^\times /\ker(\omega)$. 
Fix an
irreducible summand $\pi_1$ of the principal series representation $\Ps(1,\omega,\cdots,\omega^{n-1})$ 
of $\SL_n(k)$.
For $\omega$ a ramified character, the corresponding character of the 
Hecke algebra $\Hecke(G,\rho)$, corresponds to a character ---call it $\chi_0$---
of $A \rtimes Z/n$.
Then the representation of $\Hecke(G,\rho)$ corresponding to the character $\chi_0 \cdot \chi$ 
of $A \rtimes \Z/n$ is the same as the one corresponding to $\pi_{e(\chi)}$.
In the unramified case,
if $\pi_1$ corresponds to $\Ind_{A \rtimes \Z/n}^{A\rtimes S_n} (\chi_0)$, then 
$\pi_{e(\chi)}$ corresponds to $\Ind_{A \rtimes \Z/n}^{A\rtimes S_n} (\chi_0 \cdot \chi)$.
\end{prop}

The proof of this proposition depends on the following simple lemma, whose proof is omitted.

\begin{lemma}
\label{lem:fin-cyclic-gp-alg}
Let $C$ be a finite cyclic group of order $n$, and $\omega$ a character $C\rightarrow \C^\times$. 
Then $\omega$ extends to a character $\tilde{\omega} \colon  \Z[C] \rightarrow \C^\times$ by sending an 
element $c$ of $C$ to $\omega(c)$.
The restriction of $\tilde{\omega}$ to the augmentation ideal $\Z[C]^0$ is invariant
under the translation action of $C$ on $\Z[C]^0$.
Thus it extends to a character, say $\tilde{\omega}_0$,
of $\Z[C]^0 \rtimes C$.
Since $\Z[C]^0 \rtimes C$ is a normal subgroup of $\Z[C] \rtimes C$, 
there is an action of $ [\Z[C] \rtimes C]/ [\Z[C]^0 \rtimes C] = \Z$ on $\Z[C]^0 \rtimes C$, and 
hence on its character group.
Under this action, the element $d \in \Z$ takes $\tilde{\omega}_0$ to $\tilde{\omega}_0 \cdot \omega^d$ 
where $\omega^d$ is a character of $C$ thought of as a character of $\Z[C] \rtimes C$.
\end{lemma}

\begin{proof}[Proof of Proposition \ref{prop:compatibility}]
In both the ramified and unramified cases, we will embed our Hecke algebra $\Hecke(G,\rho)$ for $\SL_n(k)$ 
into a similar Hecke algebra
for $\GL_n(k)$.

In the case where $\omega$ is totally ramified, the type $(K,\rho)$ for $\SL_n(k)$ has a natural
variant for $\GL_n(k)$ with the type $(k^\times\cdot K, \rho')$, 
where $\rho'$ is the extension
of the representation $\rho$ of $K$ to $k^\times \cdot K$ by using the central character of the
principal series representation $\pi$ on $k^\times$.

In the case where $\omega$ is unramified,
consider the chain of groups $\SL_n(k) \subset k^\times \cdot \SL_n(k) \subset \GL_n(k)$,
and the corresponding Iwahori subgroups
$\I \subset \OO^\times\cdot \I \subset \tilde\I$.
We can embed the Iwahori algebra $\Hecke(\I)$ of $\SL_n(k)$
into the analogous algebra $\Hecke( k^\times \cdot \SL_n(k)  \doubleslash \OO^\times\cdot \I)$.
We will then compare this latter Hecke algebra with 
the Iwahori-Hecke algebra $\Hecke(\tilde\I)$
of $\GL_n(k)$.

Recall that instead of considering $\Hecke(\I)$ we are considering $\Hecke(\I)_\delta$,
obtained by inverting an element $\delta$
of its center, which can be related
to the group algebra of $A \rtimes S_n$.
A similar assertion for $\GL_n(k)$ allows one to turn
questions on Hecke algebras for $\GL_n(k)$ to one on affine Weyl group for $\GL_n(k)$.

The affine Weyl groups for $k^\times \cdot \SL_n(k)$ and $\GL_n(k)$,
parametrizing the double cosets 
$$
\I \backslash (k^\times \cdot \SL_n(k)) /{\OO}^\times\cdot \I
\qquad\text{and}\qquad
\tilde{\I} \backslash \GL_n(k)/ \tilde{\I},
$$
respectively, can be identified with
$$
(A+ \Delta\Z) \rtimes S_n
\qquad\text{and}\qquad
\Z^n \rtimes S_n,
$$
respectively,
where $\Delta \Z$ denotes the image of $\Z$ under the diagonal embedding
$\Delta \colon \Z \longrightarrow \Z^n$.
Consider the short exact sequence
$$
0\longrightarrow (A + \Delta \Z) \rtimes S_n
\longrightarrow \Z^n \rtimes S_n \longrightarrow \Z/n \longrightarrow 0, \eqno{(*)}
$$
with the natural map from $\Z^n \rtimes S_n$ to $\Z$ being the sum of co-ordinates on $\Z^n$.
Thus there is a natural action of $\Z^n \rtimes S_n$ on $ A \rtimes S_n$ via inner-conjugation,
hence on irreducible representations of $A\rtimes S_n$ 
by inner-conjugation, giving rise to
an action of $\Z/n$ on irreducible representations of $A\rtimes S_n$.

The proof of the proposition in the unramified case now follows from Lemma \ref{lem:fin-cyclic-gp-alg},
applied to the 
exact sequence
$$
0\longrightarrow (A + \Delta \Z) \rtimes \Z/n \longrightarrow \Z^n \rtimes \Z/n 
\longrightarrow \Z/n \longrightarrow 0,
$$
which is the restriction of the exact sequence $(*)$ to the subgroup $\Z/n$ inside $S_n$.
We leave the details, as well as the case of ramified character, to the reader.
We only add that in the ramified case 
one identifies the Hecke algebra for $\GL_n(k)$ for the type
$(k^\times\cdot K, \rho')$ mentioned earlier in the section 
to the group algebra of $\Z^n \rtimes \Z/n$
such that the previous short exact sequence applies, and together with Lemma \ref{lem:fin-cyclic-gp-alg},
gives the proof of the
proposition.
\end{proof}

\end{document}